\RequirePackage[warning,log]{snapshot}
\documentclass{amsart}
\usepackage{graphicx}

\newtheorem{theorem}{Theorem}[section]
\newtheorem{lemma}[theorem]{Lemma}
\newtheorem{conj}[theorem]{Conjecture}
\theoremstyle{definition}

\newtheorem{cor}[theorem]{Corollary}

\theoremstyle{remark}
\newtheorem{remark}[theorem]{Remark}

\numberwithin{equation}{section}

\begin{document}

\title{Chebyshev polynomials and a refinement of the local residue/non-residue structure at a prime}
\author{Kok Seng Chua}
\email{chuakkss52@outlook.com}





\begin{abstract} The basic power function $t_n(x)=x^n$ is in some sense a classical limit for large $x$, of the monictised  Chebyshev polynomial of the first kind $T_n(x)/2^{n-1}$. A theorem of Ritt says they are the only two families of polynomials $p_n(x)$ over $\mathbb{C}$ which satisfies the commutativity relation $p_n(p_m(x))=p_m(p_n(x))$. The commutativity $t_n(t_m(x))=t_m(t_n(x))$ is the reason why the RSA scheme allow also digital signature but the Diffie-Hellman key exchange protocol depends only on the commutativity. The DH scheme and many results in elementary local (at a fixed prime) multiplicative number theory is about properties of the power function $t_n(x)$ and they have natural analogue extension  to $T_n(x)$. Recently we discovered a Chebyshev version of Euler's primality criterion , which however depends on two quadratic characters $\epsilon_p(a)=\left ( \frac{a^2-1}{p} \right)$ and $\delta_p(a)=\left( \frac{2(a+1)}{p} \right)$. This gives rise to a local partition of $(\mathbb{Z}/p\mathbb{Z}) \setminus \{ \pm 1 \}$ into  4 disjoint sets $A_{\epsilon \delta}$. This can be thought of as a real refinement of the residue/non-residue as it arise from viewing $T_n(x)$ is the "real" part of the $n$th power of the unit $\omega_x=x+\sqrt{x^2-1}$, namely  $\omega_x^n=T_n(x)+U_{n-1}(x)\sqrt{x^2-1}$. There are obvious analogue of Chebyshev version of pseudoprimes, Wieferich primes, Lucas-Lehmer, AKS, Diffie-Hellman, cyclotomic expansions and probably others.

\end{abstract}
\maketitle

\section{Introduction}  The Chebyshev polynomial of the first kind $T_n(x)$ which one usually thought of as the polynomial which express $\cos n \theta $ in term of $\cos \theta$ ie. $\cos (n \theta)=T_n(\cos \theta)$, is well known to be useful in numerical analysis, since its roots which are all real gives optimal weights for numerical integration. As a polynomial it seems unusual that $T_n(x), T_n(x) \pm 1$ are all totally real which follows from the $\cos \theta$ substitution that $T_n(x)-t$ has only real roots for $|t| \le 1 $ and that when $t=\pm1$, all the roots are double roots, which explain the square factor in the real cyclotomic expansion later.

One can also think of $T_n(x)$ in a more arithmetical way as the "real" part of the $n$th power of the positive unit $\omega_x:=x+\sqrt{x^2-1}$, namely
\begin{align} \omega_x^n=T_n(x)+U_{n-1}(x)\sqrt{x^2-1},\end{align}

and this explains why there is also a Chebyshev polynomial of the second kind $U_{n-1}(x)$ as the "imaginary" part. It also suggests a dual version using instead of $\omega_x$ the negative unit
$\theta_x:=x+\sqrt{x^2+1}$,
\begin{align} \theta_x^n=S_n(x)+V_{n-1}(x)\sqrt{x^2+1}.\end{align}
The dual "bosonic" Chebyshev polynomials $S_n(x),V_n(x)$ are in fact just $T_n(x), U_n(x)$ with all negative coefficients changed to positive and they seems useful for explicit diophantine approximations. For example $V_n(x)$ are found to give explicit solution to a special case of Oppenhiem conjecture (now Margulis theorem). see section below

Since $\omega_x$ is a positive unit, so is $\omega_x^n$ so that its norm gives
\begin{align} \omega_x^n \overline{ \omega_x^n}=T_n^2(x)-U_{n-1}^2(x)(x^2-1)=1, \end{align}
which is the Pell's equation. (1.1) and (1.3) implies another interesting characterization of $T_n(x)$,
\begin{align} \omega_x^n=T_n(x)+\sqrt{T_n^2(x)-1}=\omega_{T_n(x)}
\end{align}
so that we have
\begin{align} \omega_{T_{mn}(x)}=\omega_x^{mn}=\omega_x^{nm}=(\omega_x^n)^m=(\omega_{T_n(x)})^m=\omega_{T_m(T_n(x))}
\end{align}
and this implies $T_m(T_n(x))=T_{mn}(x)$ and hence the compositional commutativity
\begin{align} T_m(T_n(x))=T_{mn}(x)=T_{nm}(x)=T_n(T_m(x)).\end{align}

 A theorem of Ritt states that $T_n(x)$ and the simpler basic power function
 \begin{align} t_n(x)=x^n \end{align}
 are the only two families of polynomials over $\mathbb{C}$ which satisfy the above commutativity (1.6) and they also agree locally
 \begin{align} T_p(x)=t_p(x) \mod p \end{align} at any odd prime $p$. In fact we have for $n>2$,

 \begin{align}
     T_n(x)=x^n \mod n   \Leftrightarrow \text{ n is a prime},
     \end{align}
     which suggests there should be a Chebyshev version of AKS.

 Most of the constructions/results of elementary multiplicative number theory at a fixed prime $p$ are based on properties of basic power function and its commutativity (1.7),(1.6), for example pseudoprimes, AKS, cyclotomic expansion, local quadratic residue/non-residue structure,  Diffie-Hellman key exchange protocol and Artin's primitive root conjecture and all these have a natural Chebyshev version. In particular the commutativity of $t_n(x)$ is the reason why the RSA public-key encryption scheme also allow digital signature \cite{RSA}. However a Chebyshev version of RSA seems impossible since one also need a homomorphism property $t_{n+m}(x)=t_n(x)t_m(x)$ which does not hold for Chebyshev. However one clearly have a Diffie-Hellman key exchange analogue which depends only on the commutativity and also an analogue of AKS \cite{AKS} for prime testing.

 Since $T_n(x)=\frac{\omega_x^n+\overline{\omega_x}^n}{2}=\frac{(x+\sqrt{x^2-1})^n+(x-\sqrt{x^2-1})^n}{2}$ and for large $x$, $\sqrt{x^2-1} \sim x$, we have for large $x$, $T_n(x) \sim (2x)^{n}/2$ so that the monic $T_n(x)/2^{n-1} \sim x^n$ so that in some sense $t_n(x)$ is a "classical limit" of the monictized $T_n(x)$ for large $x$.

  \section{A Chebyshev-Euler primality criteria}

  Our initial observation is that the Lucas Lehmer iteration $s_0=4, s_{n+1}=s_n^2-2$ for testing Mersenne prime can be solved (interpolated) using  Chebyshev polynomial: $s_n=2T_{2^n}(2)$. This leads us naturally to seek similar primality test based on iterating $T_{q^n}(a)$ and note that the commutativity implies one can compute $T_{q^n}(a)$ efficiently as $T_q^n(a)$ where the exponent means compositional iteration of the simpler function $T_q(x)$.

 Lucas-Lehmer is based on the arithmetic generated by $\omega_2=2+\sqrt{3}$. Replacing  $\omega_2$ by $\omega_a=a+\sqrt{a^2-1}$ led us to a Chebyshev version of Euler's primality criteria :"$t_{(p-1)/2}(a)=\left( \frac{a}{p} \right) \mod p$".

 Here the character $\epsilon(a)$ comes from our $q-ary$ extension of  Lucas -Lehmer \cite{C}, based on numerical observation. The second character $\delta(a)$ comes out when one follows the proof of the $2$-ary Lucas-Lehmer necessity condition for primality of Mersenns primes \cite{R}. One just replace $\omega_2=2+\sqrt{3}$ by $\omega_a$ and follow the argument carring the character $\epsilon(a)$ along, and the character $\delta(a)$ appears in the final step just after (2.12) in the proof below.

\begin{theorem}
For an odd prime $p$, and an integer  $$a \in R_p:=(\mathbb{Z}/p)/\{ \pm 1\}=\{0,2,3,...,p-2\},$$ and $\epsilon=\left( \frac{a^2-1}{p} \right)$, $\delta=\left( \frac{2(a+1)}{p} \right) $, we have
\begin{align} \omega_a^{\frac{p-\epsilon}{2}}=\delta, \;\;\; \omega_a^{\frac{p+\epsilon}{2}}=\delta \omega_a^\epsilon \mod p,\end{align}
or equivalently

\begin{align}T_{(p-\epsilon)/2}(a)=\delta,\;\; U_{(p-\epsilon)/2-1}(a)=0 \mod p,\end{align}
\begin{align} T_{(p+\epsilon)/2}(a)=\delta a ,\;\;  U_{(p+\epsilon)/2-1}(a)=\delta \epsilon\mod p \end{align}
$(2.2)$ also implies
\begin{align}
T_{(p-\epsilon)/2}(a)=\delta \mod p^2 .\end{align}
and
\begin{align} T_{(p-\epsilon)/2}(a)=\delta \mod p  \Rightarrow U_{(p-\epsilon)/2-1}(a)=0 \mod p. \end{align}
so that the second condition in $(2.2)$ is unnecessary.

\end{theorem}

This lead us to a natural local partition of $R_p$ similar to the quadratic residue /nonresidue $Q_{\pm}=\left\{ a \in (\mathbb{Z}/p)^*: \left( \frac{a}{p} \right)= \pm 1 \right\},$ but we now have four sets since we now have two quadratic characters $\epsilon(a), \delta(a)$.

For $\epsilon, \delta= \pm 1$, we  define 4 residue sets
\begin{align} \;\;\;\; A_{\epsilon\delta}:=\left\{ a \in R_p:\epsilon(a):= \left( \frac{a^2-1}{p} \right) = \epsilon, \delta(a):= \left( \frac{2(a+1)}{p} \right)=\delta \right\},\end{align}
whose elements must satisfies the four relations in the theorem. Also we have for any $a \in A_{\epsilon, \delta}$,

$$\left( \frac{a+1}{p} \right)=\left( \frac{2}{p} \right) \delta, \;\; \left( \frac{a-1}{p} \right)=\left( \frac{2}{p} \right)\epsilon \delta ,$$
so that shifts $A_{\epsilon, \delta} \pm 1$ must be all inside $Q_+$ or $Q_-$. Conversely a $\pm 1$ shift of $Q_{ \pm}$ is a essentially union of a "real" $A_{+.\delta}$ and a "unreal"  $A_{-,\delta}$ set,  so this is a refinement of the local quadratic structure. In fact the real $a$ defined by $\epsilon=1$ can be read off from $Q_{\pm} \pm 1$ as elements such that $-a$ also lies in the same $Q_{\pm} \pm 1$ while the "unreal" $a$ are interchanged.

One can always get a partition into four sets using any two quadratic characters but they won't be natural and one can't expect the nice properties here. Note by Theorem 2.1 one also has

\begin{align}A_{\epsilon \delta}=\left\{ a \in R_p : T_{\frac{p-\epsilon}{2}}(a)=\delta \mod p \right\},\end{align}
which is the analogue of $Q_{\pm}=\left\{ a \in \mathbb{Z}/p : t_{(p-1)/2}(a)= \pm 1 \mod p \right\}$.

\subsection{Chebyshev-Wieferich primes}
Theorem (2.1) implies that if $p$ is an odd prime and $a \not= 0,\pm 1 \mod p$, then
$$\omega_a^{(p-\epsilon)/2}= \delta \mod p.$$
This an exact analogue of little Fermat $t_{p-1}(a)=a^{p-1} = 1 \mod p$ so that it is natural to define a prime $p$ to be a Chebyshev-Wieferich prime to the base $a$ if
$$\omega_a^{(p-\epsilon)/2}= \delta \mod p^2.$$

This is equivalent to $T_{(p-\epsilon)/2}(a)=\delta \mod p^2$ and $U_{(p-\epsilon)/2-1}(a)=0 \mod p^2$, but by (2.4), we know the first condition already holds for any prime $p$ so that  $p$ is a Chebyshev-Wieferich prime  if and only if $U_{(p-\epsilon)/2-1}(a)=0 \mod p^2$. Also a Chebyshev-Wieferich prime is automatically a CW-plus prime, namely $U_{(p-\epsilon)/2-1}(a)= 0 \mod p^2$ implies
$$ \omega_a^{\frac{p-\epsilon}{2}}=\delta, \;\;\; \omega_a^{\frac{p+\epsilon}{2}}=\delta \omega_a^\epsilon \mod p^2,$$
Chebyshev-Wieferich primes seem rarer than the classical Wieferich primes. It seems thay are essentially equivalent to Wieferich prime for some Lucas sequence . The table below gives some Chebyshev- Wieferich primesup to $10^8$.
\begin{center}
\begin{tabular}{ l|l}
a & p \\
 \hline
2 & 103 \\
3 & 13,31,1546463 \\
4 & 181,1039,2917,2401457 \\
5 & 7,523 \\
6 & 23, 577, 1325663 \\
7 & 103 \;\;  $7=T_2(2)$ \\
8 & - \\
9 & -\\
10 & - \\
11 & - \\
12 & 5,311,3286453 \\
13 & 5,43, 71 \\
14 & 557 , 19739\\
15 & 6707879, 93140353 \\
16 & 5231 ,6491, 30071 \\
17 & 13, 31,1546463 \;\; $17=T_2(3)$ \\
18 & 11,3533729 \\
 \hline
\end{tabular}
\end{center}
It seems we have seen the prime 103, 1039 before (especially 1039 because it is not quite 1093 as Wieferich primes of probably Lucas 

\subsection{Characterization of $A_{\epsilon, \delta}$ in term of its multiplicative order}

Every $a \in R_p$ has a $\omega_a$ multiplicative order $ord_p(\omega_a)$ which is the least integer $n$ such that $\omega_a^n=1 \mod p$. It follows as is usual that $\omega_a^m=1 \mod p$ for any $m$ implies $ord_p(\omega_a)|m$ . Since $\omega_a^{(p-\epsilon)/2}=\delta= \pm 1 \mod p$, we have $ord_p(\omega_a)|(p-\epsilon)/2$ if $\delta=1$ and $ord_p(\omega_a)|(p-\epsilon)$ if $\delta=-1$. This gives us a most natural way to characterize the $A_{\epsilon, \delta}$ set of union of elements with some order.

By definition, $ord_p(\omega)$ is the least $n$ such that $T_n(a)=1, U_{n-1}(a)=0 \mod p$. But by the Pell's equation $T_n(a)^2-1=(a^2-1)U_{n-1}(a)^2$, $T_n(a)=1 \mod p$ implies $U_{n-1}(a)=0 \mod p$.
$T_n(a)$ also satisfies a second order linear recurrence
$$ T_0(a)=1,T_1(a)=a,T_{n+1}(a)=2aT_n(a)-T_{n-1}(a),$$ so this must have a period $\mod p$ which is the least $n$ such that $T_n(a)=1, T_{n+1}(a)=a \mod p$. However the second condition $T_{n+1}(a)=a \mod p$ is again automatically implied because of the identity
$$T_{n+1}(x)=xT_{n}(x)-(1-x^2)U_{n-1}(x).$$
So the multiplicative order always equals the additive period, both given by just the least $n$ with $T_n(a)=1$.

 For example $0$ always has multiplicative order 4 since $\omega_o=\sqrt{-1}$ and the recurrence sequence is $T_1(0)=0,T_2(0)=-1, T_3(0)=0,T_4(0)=1$. Note also $\epsilon(0)=\left( \frac{-1}{p} \right), \delta(0)=\left( \frac{2}{p} \right)$. So $0$ is real if and only if $p=1 \mod 4$.
\begin{theorem}
If we let $I_d$ be the elements in $R_p$ of order/period $d$, we have
\begin{align} A_{\epsilon +}=\cup_{\substack{d|(p-\epsilon)/2 \\ d>2}} I_d,\;\;\; A_{\epsilon -}= \cup_{\substack{d|(p-\epsilon)\\ d \not | (p-\epsilon)/2 \\d>2} }I_d,\end{align}
and $I_d$ consists exactly of the $\phi(d)/2$  roots of $\Phi^+_d(2x) \mod p$, where $\Phi_d^+(x)$ is the polynomial of the maximal real cyclotomic subfield of $\mathbb{Q}(e^{2 \pi i/d})$.
\end{theorem}

\begin{remark} An immediate consequence of our characterization of $I_d$ is that for $d|(p \pm 1)$, the polynomials $\Phi_d^+(x)$ must split completely mod $p$. In fact these are the only primes for which $\Phi_d^+(x)$ split completely mod $p$. This is analogues to the classical case : the cyclotomic polynomial $\Phi_d(x)$ split completely mod $p$ if and only if $d|(p-1)$ and the $\phi(d)$ roots of $\Phi_d(x) \mod p$ are exactly the elements in $(\mathbb{Z}/p)^*$ with multiplicative order $ord_p(a)=d$.
\end{remark}
\begin{remark}
(2.2) is an analogue of the fact that the quadratic non-residue are the elements in $(\mathbb{Z}/p)^*$ whose order $ord_p(a)$ in $\mathbb{Z}^*$ does not divide $(p-1)/2$. One main difference is that for a fixed $p$, the group $(\mathbb{Z}/p)^*$ is the same for all $a$ in the classical case, where as the ring $\mathbb{Z}/p[\sqrt{a^2-1}]$ varies with $a$ so that $R_p$ behave more like a "moduli space" of the quadratic orders $\mod p$.
\end{remark}

\subsection{Proof of theorem 2.1} We have computing mod $p$, by Euler’s criterion
\begin{align}
(a-1+\sqrt{a^2-1})^p=(a-1)+(a^2-1)^{(p-1)/2}\sqrt{a^2-1}=a-1+\epsilon \sqrt{a^2-1}
\end{align}
Multiplying by $a-1-\epsilon\sqrt{a^2-1}$ gives (for $\epsilon= \pm 1$)
\begin{align} (2-2a)^{(1-\epsilon)/2}=(a-1+\sqrt{a^2-1})^{p-\epsilon}\end{align}
Now use $( a-1+\sqrt{a^2-1})^2=2(a-1)\omega_a$ to get
\begin{align}(2(a-1)\omega_a)^{(p-\epsilon)/2}=(2-2a)^{(1-\epsilon)/2}=(-1)^{(1-\epsilon)/2}(2a-2)^{(1-\epsilon)/2} \end{align}
Now we note that $(-1)^{(1-\epsilon)/2}=\epsilon$, for $\epsilon= \pm 1$ so that
\begin{align} (2(a-1))^{(p-1)/2}\omega_a^{(p-\epsilon)/2}=\epsilon=\left( \frac{a^2-1}{p}\right)\end{align}
and this gives us
$$(2.1a)\;\;\; \omega_a^{(p-\epsilon)/2}= \left( \frac{2(a+1)}{p} \right):=\delta,$$
and is how we are lead to the character $\delta$.

(2.2) and (2.3) follows from (2.1) by taking “real and imaginary” part in $\mathbb{Z}[\omega_a].$

We also have with $n=(p-\epsilon)/2$ and the Pell's equation (1.3),
we have $ p|U_{n-1}(a)$ implies
$p^2|T_n(a)^2-1=(T_n(a)-\delta)(T_n(a)+\delta)$
but p does not divide $T_n(a)+\delta$, so (2.2) implies (2.4).

Also by (1.3) $T_n(a)=\delta \mod p$ implies $U_{n-1}(a)=0 \mod p$ since $a \not= \pm 1$.

$\square$

\section{Some Chebyshev analogue of classical elementary number theory}
Most part of elementary number theory is based on property of the power function $t_n(x)=x^n$. There is  often a Chebyshev version replacing $t_n(x)$ by $T_n(x)$, which may give some hint on the original problem.

For any $a \in R_p$, we have $\omega_a^{(p-\epsilon)/2}=\delta$ so that $\omega_a^{p-\epsilon}=1$ just like $t_{p-1}(a)=1 \mod p$.

We will say that $a$ is a Chebyshev quadratic residue mod $p$ if $\delta=1$ and a non-residue if $\delta=-1$. Also $a$ is real mod $p$ if $\epsilon=1$ and non-real otherwise.

\subsection{Chebyshev pseudoprimes} Since $T_p(a)=t_p(a)=a \mod p$ \cite{MTW} called a odd  composite integer $n$ a weak Chebyshev pseudoprime to the base $a$ if $T_n(a) = a \mod n$ . This is a weak notion and we only mention this because the Ramanujan taxi-cab number $n=1729$ can be characterized as the least composite which is a simultaneous Fermat and weak Chebyshev pseudoprime to the base $2$ ie, the least composite $n$ with $n$ dividing both  $2^n-2$ and $T_n(2)-2$. $1729=7 \cdot 13 \cdot 19$ is also the 3rd smallest Carmichael number.

If $p$ is a prime and $a$ an integer with $gcd(a^2-1,p)=1$, by Theorem 2.1, we have
\begin{align}T_{\frac{p-\epsilon(a)}{2}}(a)=\left( \frac{2(a+1)}{p} \right),\;\; U_{\frac{p-\epsilon(a)}{2}-1}(a)=0\;\;\; mod \;\;p.
\end{align}

We shall now called an odd non-prime integer $n$, with $gcd(n,a^2-1)=1$, which pass (3.1) a Chebyshev pseudoprime to the base $a$. It depends only on $a$ mod $n$ but there is no subgroup structure. Chebyshev pseudoprimes are always squarefree except for some prime squared.They are rare and seems rarer than Fermat pseudoprimes. There are only seven of them to the base $2$ upto $20000$,
$$23.43,\;\;\;37.73,\;\;\;103^2,\;\;\;61.181,\;\;\;5.7.443,\;\;\;97.193,\;\;\;31.607.$$
 Is there a Chebyshev pseudoprime to every base mod $n$ ? A Sierpi\'nski number  is a positive odd integer $k$ such that $N_n=k2^n+1$ is composite for every $n \ge 1$. $k_0=78557$ is the smallest known Sierpi\'nski number, because every $N_n=k_02^n+1$ is divisible by one of $\{3, 5, 7, 13, 19, 37, 73\}$. It may be possible that $N_n$ fail a Chebyshev test for every $n$ for some $a$. Since $N_n=1$ mod $8$ for $n \ge 3$. We get $\epsilon=\delta=1$ if we pick $a=3$. So if $s_0=3, s_{k+1}=2s_k^2-1$, and $N_n^2=(k2^n+1)^2$ does not divide $T_k(s_{n-1})-1$ for every $n \ge 3$, then $k$ is Sierpi\'nski. Note $N_{n+1}=2N_n-1$.
It is open if any of the following five numbers $21181, 22699, 24737, 55459, 67607$ is Sierpi\'nski.

A Chebyshev pseudoprime for the base $a$ is  also a weak Chebyshev pseudoprime  since the condition on $U$ means $\omega^{n-\epsilon}=1$ or $\omega^n=\omega^\epsilon$ and taking trace gives $T_n(a)=T_1(a)=a$ mod $n$.

Note that a square-free $n$ which pass the $T$ test will also pass the $U$ test.[Proof: We have $T_{p-\epsilon}(a)=1$ so that
$(\omega^{(n-\epsilon)/2}-\overline{\omega}^{(n-\epsilon)/2})^2=0$ and squarefeeness of $n$ implies $U_{(n-\epsilon)/2-1}(a)=0$.]
There are many non square-free integers which pass the $T$ test but the only non squarefree integer which can pass both tests are square of prime. So  the second part only serve to rule out non squarefree integer and this is relevant since there is no known efficient  algorithm to detect squarefreness. However we can always rule out perfect square as input

If $(n-\epsilon)/2=2^tQ_1$ is even , we can look at the profile
$$[T_{Q_1}(a), T_{2Q_1}(a),T_{2^2Q_1}(a),...,T_{(n-\epsilon)/2}(a)] \mod n$$
 as in the strong Fermat pseudoprime test. Since $T_2(x)=2x^2-1$,if  there is a $1$ not preceded by $\pm 1$ or a $-1$ not preceded by $0$,
$n$ cannot be prime. For the seven pseudoprimes above, the profiles are
$$[1],[0,-1],[9083,0,-1,1],[0,-1,1,1,1],[8416,4431,8861,1],$$
$$[14063,17370,18527,387,1],[18791,1301,18720,0,-1,1],$$
so the strong test rule out $5.7.443$ and $97.193$ as primes. For square free $n$, $-1$ is always preceded by $0$, since $(\omega^m+\overline{\omega}^m)/2=-1$ implies $(\omega^{m/2}+\overline{\omega}^{m/2})^2=0$ mod $n$.

We note that for $a \neq 0 \pm 1$, $(\omega_a,\overline{\omega}_a)$ forms a Lucas pair in the sense of \cite{BHV}, since
$\frac{\omega_a}{\overline{\omega}_a}$ is not a root of unity. The associated Lucas number $u_n(\omega_a,\overline{\omega}_a)=U_{n-1}(a)$. It seems to follow from \cite{BHV} that for every $n>1$, $U_n(a)$ has a primitive divisor, ie. there is a prime $p$ which divides $U_n(a)$ but not $a(a^2-1)U_0(a)...U_{n-1}(a)$.

\subsection{Chebyshev-Artin primitive root conjecture} Artin primitive-root conjecture states that an integer $a \not= -1$ and not a perfect square is a primitive root mod $p$ for infinitely many $p$. A Chebyshev primitive root is an element in $a \in R_p$ with $\omega_a$ order $p-\epsilon$. Just like $t_2(x)=x^2$ are perfect square, a Chebyshev square is an integer of form $a=T_2(x)=2x^2-1$. Chebyshev squares are always residue since $$\omega_{T_2(x)}^{(p-\epsilon)/2}=(\omega_x^2)^{(p-\epsilon)/2}=\omega_x^{p-\epsilon}=1$$ so they cannot be primitive.

Note Chebyshev squares are $7,17,31,49,71,97,127,...$. $0$ is never a primitive root since it always has order $4$. It plays the same role as $-1$ in the classical case which always has order two. The analogue of Artin's conjecture is

\begin{conj} Every integer $a \not= 0, \pm 1, 2n^2+1$, and  not a Chebyshev squares is a Chebyshev primitive root mod $p$ for infinitely many $p$.
\end{conj}
We compute that such $a$ actually has primitive roots and list the least $p$ so that $a$ is a  (both real/unreal) primitive root mod $p$.
For $a=3,9$ and in general $a=2n^2+1,core(a^2-1)=core(2(a+1))=n^2+1)$ we have $\epsilon(a)=\delta(a)$ so that there are no real ($\epsilon=1$) primitive root as this implies $a$ is a residue $\delta=1$.
Numerical computations shows there are always many such primitive roots. One should be able to prove a density estimates based on probabilistic argument as Artin did and the one later adjustment. Is the least primitive roots still $o(p)$ and can the bound still be improved assuming GRH ?
\begin{center}
\begin{tabular}{ l|l|l l|}
a & p-1 & p+1 \\
 \hline
2 & 11 & 7 \\
3 & - & 3 \\
4 & 7 & 19 \\
5 & 5 & 7 \\
6 & 17 & 3 \\
8 & 19 & 5 \\
9 & - & 3 \\
10 & 5 & 17 \\
 \hline
\end{tabular}
\end{center}
Maybe we should define a primitive root to be an $a$ with $\omega_a$ order $p+1$.

\subsection{Chebyshev cyclotomic expansion} One of the most useful algebraic identity is $t_{2n}(x)-1=(t_n(x)-1)(t_n(x)+1)$. This still holds for Chebyshev because $T_2(x)=2x^2-1$, so that
$$T_{2n}(x)-1=T_2(T_n(x))-1=2T_n(x)^2-2=2(T_n(x)-1)(T_n(x)+1)$$
This is just a special case of a more general version of the Chebyshev version of the cyclotomic expansion
$$t_n(x)-1=\prod_{d|n} \Phi_d(x),$$
where $\Phi_d(x)$ is the irreducible cyclotomic polynomial. We have
\begin{theorem} (Chebyshev-cyclotomic expansion) Let $\Psi_1(x)=x-1, \Psi_2(x)=2(x+1), \Psi_d(x)=\Phi_d^+(2x)^2,$
$$T_n(x)-1=\prod_{d|n} \Psi_d(x)$$
Note the roots of $\Psi_d(x)$ are primitive $d$th roots of $T_d(x)-1$, those which are not also roots of $T_k(x)-1,k<d$.

A corollary is this
\begin{lemma} Let $a \not =0, \pm 1, n \ge 1$ if $T_n(a) \not= \pm 1,$ then $\epsilon(T_n(a))=\epsilon(a)$. So if $a$ is real/unreal then so is all the $T_n(a)$. Also we have $\delta(T_n(a))=1$ if $n$ is even and $\delta(T_n(a))=\delta(a)$ for odd $n$. So $\delta(a)=1$ implies $\delta(T_n(a))=1$ for all $n$. $\delta(a)=-1$ implies $\delta(T_n(a))$ alternate in sign.
\end{lemma}

\begin{proof} We have
$$T_n(x)^2-1=\frac{1}{2}(T_{2n}(x)-1)=\frac{1}{2} \Psi_1(x)\Psi_2(x)\prod_{d|2n, d>2}\Phi_d^+(2x)^2,$$
so that $\left( \frac{T_n(a)^2-1}{p} \right)=\left( \frac{a^2-1}{p} \right)$. This also follows from $T_n(x)^2-1=(x^2-1)U_{n-1}(x)^2$ from the Pell's equation. For $\delta$ we have for $n=2^mt$
$$T_n(x)+1=\frac{1}{2}\frac{T_{2n}(x)-1}{T_n(x)-1}=\frac{1}{2} \prod_{d|t} \Psi_{2^{m+1}d}(x)$$
So for odd $n$, $T_n(x)+1=(x+1)g(x)^2$. For even $n$, $2(T_n(x)+1)=h(x)^2$.

\end{proof}
\end{theorem}

There ia also a refined version of the cyclotomic expansion, since
$$A_{\epsilon \pm}=\left\{ a \in R_p: T_{(p-\epsilon)/2}(a) = \pm1 \right\},$$
(2.8) above is equivalent to for $\epsilon= \pm 1$

$$T_{(p-\epsilon)/2}(x)-1=(x-1)(2x+2)^\lambda  \prod_{\substack{d >2 \\ d|(p-\epsilon)/2} }  \Phi^+_d(2x)^2$$
$$T_{(p-\epsilon)/2}(x)+1=((x+1)/2)^\lambda  \prod_{\substack{d >2 \\ d|(p-\epsilon)\\ d \not|(p-\epsilon)/2} }  \Phi^+_d(2x)^2,$$

\subsection{Commutativity and Chebyshev-Diffie-Hellman key exchange protocol}
Recall that a public key encryption scheme \cite{DH} use a matching pair of keys $(d,e)$. Every one can encrypt a message $x$ using the public key $e$ and an encryption algorithm $E_e(x)$ but only the holder of the secret key $d$ can decrypt by computing $D_d(E_e(x))=x$. The first and still the most successful implementation of such a PK system is the RSA \cite{RSA} scheme which use the same algorithm for $D$ and $E$ ie. $E_e(x)=t_e(x) \mod m$, $D_d(x)=t_d(x) \mod m$ where key pairs satisfies $de=1 \mod \phi(m)$ 

 Because of the commutativity $t_d(t_e(x))=t_e(t_d(x))$ we have $$E_e(D_d(x)=D_d(E_e(x))=x.$$ This means the holder of the secret key can "sign" a $x$ now thought of as a document by decrypting it to produce a signature $s=D_e(x)$ which the public can verify its authenticity by encrypting to recover the original document $x=E_e(s)=E_e(D_d(x))$. Since only the holder of the secret key can produce a signature which can be encrypted/verified to recover a meaningful $x$, this prove he/she must be the real signer. So the commutativity is the reason why RSA enable also digital signature. However it seems not possible to construct a Chebyshev-RSA scheme since we also need a homomorphism property $t_{n+m}(x)=t_n(x)t_m(x)$ which is not available for Chebyshev. Even though the above is well-known material, we cannot resist writing about it to point out the fact that commutativity is the reason that made digital signature possible and this seems to be still the most beautiful application of number theory.

However the Diffie-Hellman key exchange protocal uses only the commutativity property and seemingly the Chebyshev discrete logarithm problem : finding $n$ which satisfies $T_n(a)=m \mod p$ for a large prime $p$, given $a,m$ is also hard.

Recall the classical Diffie-Hellman key exchange protocol. A,B  choose a large prime $p$ and a primitive root $g$. A generate random key $a$ and B generate random key $b$. A compute $g_a:=t_a(g) \mod p$ and send to B while B compute $g_b :=t_b(g) \mod p$ and send to A. Now A and B share a common key $g_{ab}=g_{ba}$ which A compute as $g_{ab}=t_a(g_b) \mod p$ and  B computes as $g_{ba}=t_b(g_a) \mod p$. They now have a shared common key because of the commutativity
$$g_{ab}=t_a(t_b(g))=t_b(t_a(g)=g_{ba}.$$
Obviously we have a Chebyshev -version because we still have the commutativity $T_a(T_b(g))=T_b(T_a(g))$ and by Ritt theorem, this is the only two  possibilities. We just replace all the $t_a(g)$ above by $T_a(g)$ for a Chebyshev primitive root $g$. However it is necessary to compute $T_a(g) \mod m$ for very large $a$ and $m$ efficiently. The obvious relation $\omega_{a}^{n+1}=\omega_a \omega_a^n$ gives us the recursion :

\begin{equation}
\begin{pmatrix} T_{n+1}(a) \cr U_{n}(a) \end{pmatrix}= \begin{pmatrix} a & a^2-1 \cr a & 1 \end{pmatrix} \begin{pmatrix} T_n(a) \cr U_{n-1}(a) \end{pmatrix}=
     \begin{pmatrix} a & a^2-1 \cr 1 & a \end{pmatrix}^{n} \begin{pmatrix} a \cr 1 \end{pmatrix}\;\; mod \;\;\; m.
\end{equation}
which allow us to compute this using the usual binary exponentiation and the mod $m$ ensure we only work with number no bigger than $m$.
Also we can find Chebsyshev primitive root in random polynomial time.

However the set  $ \{ T_n(a) \mod p:n=1,2.3...\}$ is always all real or all unreal since $\epsilon(a)=\epsilon(T_n(a))$. So unlike the classical case,  a Chebyshev primitive root only generate half the set $R_p$  so this scheme may not be so useful.

\subsection{Square-root cancellations in exponential sums}
Gauss famously evaluated the quadratic Gauss sum which give square-root cancelation over the set of residues. It is natural to consider if  exponential sum over our $A_{\epsilon\delta}$ gives similar cancellation over the smaller sets. For a subset $A \subset \mathbb{F}_p$, we define

\begin{align} g_A:=\sum_{a \in A} \zeta^a,\;\;\;\; \zeta=e^{2 \pi i/p} \end{align}

Let $R,N$ be the residue/non-residue mod $p$ and also $g_{\epsilon, \delta}$ for $g_{A_{\epsilon, \delta}}$. Gauss's evaluation of the quadratic Gauss sum $\sum_{a=1}^{p-1} \left( \frac{a}{p} \right)= \epsilon_p \sqrt{p}$ is equivalent to
\begin{align} g_R=\frac{-1+\epsilon_p \sqrt{p}}{2},\;\;\; g_N=\frac{-1-\epsilon_p\sqrt{p}}{2},\;\;\; \epsilon_p=
\begin{cases} 1 & p=1 \mod 4 \\ i & p=3 \mod 4  \end{cases}
\end{align}

which allows explicit evaluation the sum or difference of two $A_{\epsilon \delta}$ for example

\begin{lemma}

$$
g_{+-}-g_{++}=
\begin{cases}   (1- \left( \frac{2}{p} \right)\sqrt{p}) \cos \frac{2 \pi}{p}, \;\; p=1 \mod 4 \\
i  \left( \sin \frac{2 \pi}{p}-\left( \frac{2}{p} \right) \sqrt{p} \cos \frac{2 \pi}{p}   \right), \;\; p=3 \mod 4

\end{cases}
$$

$$
g_{--}-g_{-+}=
\begin{cases}   i(1+ \left( \frac{2}{p} \right) \sqrt{p})\sin \frac{2 \pi}{p}, \;\; p=1 \mod 4 \\
 \cos \frac{2 \pi}{p}-\left( \frac{2}{p} \right) \sqrt{p} \sin \frac{2 \pi}{p}   , \;\; p=3 \mod 4

\end{cases}
$$

\end{lemma}

This gives square root cancelation over a set of size about $p/2$ but we still have square root cancelation over each of the smaller $A_{\epsilon  \delta}$ set of size $p/4$. Note first that summing over $R_p$ is summing over $\{0,2,3,...,p-2\}$ so we have
\begin{lemma}
$$ \sum_{a \in R_p} \zeta^a= -2 \cos \frac{2 \pi}{p},$$
$$ \sum_a \left( \frac{a-1}{p} \right)\zeta^a =\epsilon_p \sqrt{p}\zeta-\left( \frac{-2}{p} \right) \zeta^{-1},$$
$$  \sum_a \left( \frac{a+1}{p} \right) \zeta^a=\epsilon_p \sqrt{p}\zeta^{-1}-\left( \frac{2}{p} \right) \zeta$$
$$ \sum_a \left( \frac{a^2-1}{p} \right) \zeta^a=\left( \frac{-1}{p} \right)+S, \;\;\; S=\sum_{a=1}^{p-1} \left( \frac{a^2-1}{p} \right) \zeta^a$$
\end{lemma}

This allows us to evaluate each $g_{\epsilon \delta}$ explicitly except for a term given by $S$, which  gives a square-root cancelation using Weil's bound for $S$. The "trick" is since $\epsilon(a), \delta(a) \in \{ \pm 1\}$
\begin{align}
g_{\epsilon\delta}=\sum_{ a \in R_p}\frac{1+\epsilon\epsilon(a)}{2} \frac{1+\delta\delta(a)}{2}  \zeta^a
\end{align}
$$=\frac{1}{4}\sum_a \zeta^a+\frac{\epsilon}{4}\sum_a \epsilon(a)\zeta^a+\frac{\delta}{4}\sum_a \delta(a)\zeta^a +\frac{\epsilon \delta}{4}\sum_a\epsilon(a) \delta(a) \zeta^a$$
$$=\frac{1}{4} \sum_a \zeta^a+\frac{\epsilon}{4}\sum_a\left( \frac{a^2-1}{p} \right) \zeta^a+\frac{\delta}{4}\left( \frac{2}{p} \right)\sum_a \left( \frac{a+1}{p} \right)\zeta^a+\frac{\epsilon \delta}{4}\left( \frac{2}{p} \right)\sum_a\left( \frac{a-1}{p} \right)\zeta^a$$

$$=-\frac{\cos 2 \pi/p}{2}+\frac{\epsilon}{4}\left ( \left( \frac{-1}{p} \right)+S \right)+\frac{\delta}{4} \left(\epsilon_p\sqrt{p}\left( \frac{2}{p} \right)\zeta^{-1}-\zeta \right)
+\frac{\epsilon \delta}{4}\left(\epsilon_p\sqrt{p}\left( \frac{2}{p} \right)\zeta  -\left(\frac{-1}{p} \right)\zeta^{-1}\right),$$

So each of the $g_{\epsilon \delta}$ can be evaluated explicitly except for the known sum $S$.

Using the well known Weil's bound $|S| \le 2 \sqrt{p}$, we have

\begin{theorem}  For all odd prime $p$ and $\epsilon, \delta \in \{ \pm 1 \}$,
$$|g_{\epsilon \delta} | \le \sqrt{p}+5/4$$

\end{theorem}

Note also $g_{\epsilon\delta}$ is real if $\left( \frac{-1}{p} \right) \epsilon>0$ and non-real but $g_{\epsilon+}=\overline{g_{\epsilon-}}$ otherwise.


It seems the "trick" (3.4) can be applied in general to other real characters eg.  $\alpha(a), \beta(a), \gamma(a) \in \{  \pm 1 \}$ where we defined $A_{\alpha \beta \gamma}$ analogously and evaluate

$$g_{\alpha \beta \gamma}=\sum_{a \in R_p} \left(  \frac{1+\alpha  \alpha(a)}{2}  \frac{1+\beta  \beta(a)}{2}
\frac{1+\gamma  \gamma(a)}{2}    \right) \zeta^a$$
to get cancelations on a smaller $1/8$ set for suitable characters.

\subsubsection{Square root cancellation in half subset of the quadratic residue or non residue}
We can also transfer the cancellation to subsets of the residues/non-residues which may be easier to use. Each of $R \pm 1$ and $N \pm 1$ is essentially a union of two of the $A_{\epsilon \delta}$ and can be read off as just the symmetric ($a \in X \implies p-a \in X$) and its complement in the translated set. If we now translate the partitioned set back by shifting $\mp 1$, we get subsets of $R,N$ which must have the same exponential sum since they differ from the $g_{\epsilon \delta}$ by a $\zeta^{\pm1}$ multiplicative factor. For example if $p=23$ the four sets which can computed using either (2.1) or (2.2) are
$$A_{--}=[4,9,10,13,14,19],A_{-+}=[0,8,11,12,15],$$
$$A_{+-}=[6,16,18,20,21],A_{++}=[2,3,5,7,17]$$
The quadratic residue set is
$$R=[1,4,9,16,2,13,3,18,12,8,6]$$
$$R+1=[2,5,10,17,3,14,4,19,13,9,7]$$
The symmetric elements of $R+1$ is thus
$$[10,13,14,9,4,19]=A_{--}$$ while the non-symmetric elements form
$$[2,5,17,3,7]=A_{++}$$
Transfering the last two sets back by substracting 1 gives two subsets of the quadratic residues with exponential sum bounded by $5/4+\sqrt{p}$
$$SR^+=[9,12,13,8,3,18],CR^+=[1,4,16,2,6],$$
where here $SR^+$ means the symmetric subsets of $R$ partition via a $+1$ shift, and clearly reading off directly from $R$,
$$SR^+=\{ a \in R:p-2-a \in R\}, CR^+=R \setminus SR^+$$
and similarly using a negative shift,
$$SR^-=\{ a \in R:p+2-a \in R\}, CR^-=R \setminus SR^-,$$
and similarly a partition of $N$, $SN^{\pm}, CN^{\pm}$ using a $\pm 1$ shift of $N$.
We also have
$$R-1=[0,3,8,15,1,12,2,17,11,7,5]$$
whose symmetric and non-symmetric subsets are given by
$$A_{-+}=[0,8,15,12,11],A_{++}=[3,1,2,17,7,5]$$
whose translate back by adding 1 to subsets of $Q$
$$SR^-=[1,9,16,12,13],CR^-=[4,2,3,18,8,6].$$
Note that we always have $1 \in SR^-$ since $1 \in R$ and $p+2-1=1 \mod p$, and $p-1$ is always in $SR^+$ or $SN^{+}$ since $p-2-(p-1)=p-1 \mod p$.
Since the non-residues
$$N=\{ 4,6,9,10,13,14,16,18,19,20,21\},$$
we can just read off the symmetric partition
$$SN^+=\{7,14,10,11,22\}, SN^-=\{5,20,10,15,11,14\}.$$

Note $A_{++}$ appears in both the $R \pm 1$ shift and $A_{+-}$ does not appear above. It will appear twice when we shift  $N \pm 1$.
It is always the case that $R\pm1, N \pm 1$ is essentially a union of two of the $A_{\epsilon \delta}$ (drop the $1,p-1$). Since $0$ is symmetric and $\epsilon(0)=\left( \frac{-1}{p} \right)$, the symmetric sets are $A_{+1},A_{++}$ with $\epsilon=+$ when $p=1 \mod 4$ and $A_{--},A_{-+}$ when $p=3 \mod 4$.

So indeed we have a refinement of the residue/nonresidue which preserve provable square-root cancelations.

 We may also verify (2.8) for $p=23$ above to define the $A_{\epsilon \delta}$ via the multiplicative order, we must have
 $$A_{++}=I_{11},\;\;A_{-+}=I_3 \cup I_4 \cup I_6 \cup I_{12},\;\;A_{+-}=I_{22},\;\;A_{--}=I_8 \cup I_{24},$$
 and the $I_d$ are computed as the roots of splitting $\Phi_d^+(2x) \mod p$ which gives
 $$I_3=[11],I_4=[0],I_6=[12],I_8=[14,9],I_{11}=[2,3,5,7,17],I_{12}=[8,15],$$
 $$I_{22}=[6,16,18,20,21],I_{24}=[19,13,4,10]$$
 which indeed agree with (2.8). These are also the only $d>2$ (divisors of $23 \pm 1$) where $\Phi_d^+(x)$ splits completely mod $p$

 We can also verify that $ord_p(\omega_a)=24$ for the unreal primitive root $a=19$ . This order is the least $n$ such that $\omega_a^n=1 \mod p$  or the least $n$ such that $([T_n(a),U_{n-1}(a)]=[1,0] \mod p$. It also equals the period of the second order linear recurrence $T_0(a)=1,T_1(a)=a,T_{n+1}(a)=2aT_n(a)-T_{n-1}(a) \mod p$ which is the least $n$ such that $T_n(a)=1, T_{n-1}(a)=T_{n+1}(a)=a \mod p$ but it is given just by the least $n$ such that $T_n(a)=1 \mod p$ as it is automatic that $T_n(a)=1 \mod p$ implies $T_{n+1}(a)=a, U_{n-1}(a)=0 \mod p$, so we need only compute the first coordinate below. It seems interesting that the multiplicatively defined order can be computed additively as a period of a linear recurrence.

Indeed we have the iterates for $n=1,2,3,...24$ is

 $$[19,1],[8,15],[9,17],[12,10],[10,18],[0,7],[13,18],[11,10],[14,17],[15,15],[4,1],[22,0],$$
 $$[4,22],[15,8],[14,6],[11,13],[13,5],[0,16],[10,5],[12,13],[9,6],[8,8],[19,22],[1,0]$$

 As noted before even though $19$ is a primitive root only half the subset of $R_p$ appears in the first coordinate, only those with $\epsilon(a)=-1,0$ and note the first coordinate is palidormic,  as this is a negative cycle as $I_{24} \subset A_{--}$. Actually this is some kind of second order Lucas sequence for fixed $a=19$ which is well known. Our main observation is that there is a structure  to view all the $a$ together, $R_p$ is like a moduli space for all the order $(\mathbb{Z}/p\mathbb{Z})[\sqrt{a^2-1}]$. It seems useful to state our observation on the partition of the residue/non-residue.

 \begin{lemma} Let $X$ be either $R$ or $N$. Then can can read off a partition of $X$ into two disjoint sets in two distinct ways which are translates of $A_{\epsilon \delta}$ shifted by $\pm 1$. WE define

 $$SX^+=\{ a \in X : p-2-a \in X\}, CX^+=X \setminus SX^+,$$
 and
 $$SX^-=\{ a \in X : p+2-a \in X \}, CX^-=X \setminus SX^-$$
 then $$|SX^{\pm}|,|CX^{\pm}| \le 9/4 + \sqrt{p}$$

 \end{lemma}

 \subsubsection{Evaluating $h_{\epsilon \delta}:=\sum_{a \in A_{\epsilon \delta}}a$} Clearly
 $$h_{\epsilon \delta}=\sum_{a \in R_p}
 \frac{1+\epsilon \epsilon(a)}{2} \frac{1+\delta \delta(a)}{2} $$

 If we let $A_{\epsilon \delta}^{\pm}$ be the set $A_{\epsilon \delta} \pm 1$ and for any set $A$, the sum $ \sum A:=\sum_{a \in A} a$. We then have for $p=3 \mod 4$,

 \begin{align} \sum N -\sum R=p \cdot h(-p) = \left( \frac{2}{p} \right)   \left(\sum A_{--}^+ +\sum A_{+-}^+-\sum A_{-+}^+ -\sum A_{++}^+ -2 \right) \end{align}
Also,
 \begin{align} p \cdot h(-p) = \left( \frac{-2}{p} \right)   \left(\sum A_{--}^- +\sum A_{++}^--\sum A_{-+}^- -\sum A_{+-}^- -(p-2) \right)+p \end{align}

Can we prove $RHS>0$ ?
Numerical observation :  For (3.6), the term $\sum A_{+-}^+$, even $\lfloor (\sum A_{+-}^+)/p \rfloor$ dominate the other it is a unique max/min when $\left( \frac{2}{p} \right)=1/-1$.
For (3.7) , it is dominate by $\sum A_{++}^-$

\subsection{Chebyshev-AKS}

We know that $T_n$ and $t_n$ agrees locally $T_p(x)=t_p(x) \mod p$ but in fact there is a converse
\begin{lemma} For integer $n >2$,
\begin{align} T_n(x)=t_n(x)=x^n \mod n \end{align} if and only if $n$ is prime.
\end{lemma}

But we also know that for $n \ge 2, gcd(a,n)=1$,
\begin{align} t_n(x+a)=t_n(x)+a \mod n, \end{align}
if and only if $n$ is prime.
This is the basis of the AKS deterministic polynomial time prime testing algorithm \cite{AKS}. (3.8) and (3.9) implies
the Chebyshev analogue of Lemma 3.9
\begin{lemma} Given $n \ge 2 , gcd(a,n)=1,$ then $n$ is prime if and only if
\begin{align} T_n(x+a)=T_n(x)+a \mod n.\end{align}
\end{lemma}
This implies there should also be a Chebyshev version of AKS.  To prove (3.8) we need a lemma first

\begin{lemma}
(a) For $ 1 \le k \le \lfloor n/2 \rfloor$, $$(-1)^ka_k:=\sum_{t=k} ^{\lfloor n/2 \rfloor} {n \choose 2t} { t \choose k} = {n-k-1 \choose k-1} \frac{n}{k} 2^{n-(2k+1)}$$

(b) If $n$ is odd composite and $p|n$ then ${n-p-1 \choose p-1}=1 \mod p$
 \end{lemma}

\begin{proof} (a) Let
$$F(y):=\sum_{t \ge 0} {n \choose 2t} y^t =\frac{(1+\sqrt{y})^n+(1-\sqrt{y})^n}{2}$$
then
$$\sum_{t \ge k} {n \choose 2t}{ t \choose k} =\frac{1}{k!} F^{k}(1)=\frac{1}{k!}\frac{n(n-k-1)!2^{n-1}}{2^{2k}(n-2k)!}.$$
(b) Let $n=ap$, $a \ge 3 $ and odd. Let the $p$-ary expansion $m=n-p-1$ be $m=\sum_{j=0}^k m_jp^j$. Since $m=(a-2)p+(p-1)$, we must have $m_0=p-1$, by Lucas lemma.
$$ {n-p-1 \choose p-1}={p-1 \choose p-1} {m_1 \choose 0} ... {m_k \choose 0}=1 \mod p.$$
\end{proof}

\subsubsection{Proof of (3.8) }
\begin{proof} Since $T_n(x)$ is the real part of the unit $\omega_x^n$, we have
$$T_n(x)=\frac{\omega_x^n+\overline{\omega_x}^n}{2}=\frac{(x+\sqrt{x^2-1})^n+(x-\sqrt{x^2-1})^n}{2}$$
 $$= \sum_{k=0}^{\lfloor n/2 \rfloor}  {n \choose 2k}x^{n-2k}(x^2-1)^k=\sum_{k=0}^{\lfloor n/2 \rfloor}a_kx^{n-2k}$$
where $$a_k=(-1)^k\sum_{t=k}^{\lfloor n/2 \rfloor}{n \choose 2t}{t \choose k}.$$
So $a_0=  \sum_{t=0}^{\lfloor n/2 \rfloor}  {n \choose 2t}=2^{n-1}$ and for $n=p$ an odd prime then $p|{n \choose 2t}$ for $0<k \le t \le \lfloor n/2 \rfloor$ as $2 \le 2t<p$, so that $p|a_k$.  $T_p(x)=2^{p-1}x^p=t_p(x) \mod p$. 

By Lemma 3.9[1] (this is also formula (16) of \cite{Wo})

\begin{align} T_n(x)=2^{n-1}x^n +\sum_{k=1}^{\lfloor n/2 \rfloor} (-1)^k {n-k-1 \choose k-1} \frac{n}{k}2^{n-(2k+1)}x^{n-2k}.\end{align}
If $n=2t$ is an even composite, then by (3.11) $a_t=(-1)^t \not=0 \mod n=2t$.
 If $n$ is odd composite and $p$ is any prime dividing $n$, then since $p$ does not divide ${n-p-1 \choose p-1}= 1 \mod p$, $a_p \not= 0 \mod n$. Note the proof shows if $n=\prod p_j$ is odd composite squarefree, then $gcd(n,a_{p_j})=n/p_j$ so $a_{p_j} \not=0 \mod n$.
\end{proof}

\subsubsection{Chebyshev polynomial knows the factors}

By Lemma (3.8) any non-vanishing $a_k$ is a witness to the compositeness of $n$. In fact it tells us more, it actually factors $n$.
First write
$$g_n(x):=T_n(x)-t_n(x)  =(2^{n-1}-1)x^n+\sum_{k=1}^{\lfloor n/2 \rfloor} a_k x^{n-2k}$$
\begin{lemma} Assume $n$ is odd and $k>1$. Then $a_k \not=0 \mod n$ implies 
\begin{align} gcd(n,n-2k)>1, n|a_k(n-2k),gcd(a_k,n)>1.\end{align}
\end{lemma}

\begin{proof} If $gcd(n-2k)=gcd(n,k)=1$, $a_k=0 \mod n$. We also have
$$a_k(-1)^k={n-k \choose k}\frac{n}{2(n-k)}2^{n-2k}={n-k \choose n-2k}\frac{n}{2(n-k)}2^{n-2k}={n-k-1 \choose n-2k-1}\frac{n(n-k)}{2(n-2k)}$$
So we have $n|a_k(n-2k)$ since $n$ is odd. But $n \not|(n-2k)$, we must have $gcd(a_k,n)>1$.
\end{proof}

(3.12) clearly implies that any non-vanishing term $a_kx^{n-2k} \mod n$ must gives a nontrivial factorisation of $n$. In particular,
\begin{cor}
If $n=pq$, every non vanishing $a_k \mod n$ factor $n$ as $n=gcd(a_k,n)gcd(k,n)$. This works even when $p=q$
\end{cor}

For example take a convenient composite number like $n=77=7 \cdot Eleven$, we have
$$T_{77}(x)-x^{77} \mod 77 =8x^{77} + 33x^{63} + 42x^{55} + 44x^{49} + 56x^{33} + 66x^{21} + 70x^{11} + 66x^7$$
so every non-leading term factor $n=gcd(k,n)gcd(a_k,n)$.

\bibliographystyle{amsplain}

\end{document}